\documentclass[10 pt]{article}
 \usepackage{times}
\usepackage{graphicx} 
\usepackage[english]{babel}
\usepackage{amsthm}
\usepackage{amssymb}
\usepackage{fancyhdr}
\usepackage{xcolor}
\usepackage{hyperref}
\usepackage{pict2e} 
\newcommand\quotientthree[2]{{#1}\Big/{#2}}
\usepackage{url}
\usepackage{amsmath}
\usepackage[dvipsnames]{xcolor}
\usepackage{geometry}
\usepackage{mathtools}
\usepackage{tikz}
\usepackage{subfigure}
\usepackage{tikz-cd}
\usepackage{subcaption}
\usepackage{tikz}
\usetikzlibrary{positioning}
\usepackage{graphicx}
\usetikzlibrary{shapes.geometric, positioning}
\usepackage{caption}
\usepackage{subcaption} 
\usetikzlibrary{positioning,chains,fit,shapes,calc}
\date{}
\usepackage{tikz}

\title{\sc On Posets of Classes of Subgroups with same set of orders of elements }
\author{{\sc Sachin Ballal}\footnote{sachinballal@uohyd.ac.in (Corresponding author)} \hspace{0.1cm}and \sc{Tushar Halder}\footnote{24mmpp03@uohyd.ac.in} \\ \small School of Mathematics and Statistics, University of Hyderabad, 500046, India
}
\theoremstyle{definition}
\newtheorem{theorem}{Theorem}[section]
\newtheorem{corollary}{Corollary}[theorem]
\newtheorem{lemma}[theorem]{Lemma}
\newtheorem{remark}{Remark}
\providecommand{\keywords}[1]
{
	\small	
	\textbf{{Keywords:}} #1
}
\begin{document}
	\maketitle
	\begin{abstract}
		In this paper, we study the posets of classes of subgroups of finite group having same set of orders of elements. We show that this poset is a chain only in the case of $p$-groups and moreover, we characterize all finite groups for which this poset is $C_2$, the chain with two elements. We also show that this poset forms a lattice in the case of finite cyclic and dihedral groups and give a characterization when this lattice is distributive and modular. \vspace{0.2cm}\\
	\end{abstract}
	\section{Introduction}
	The theory of subgroup lattices began with Ada Rottländer \cite{Ada} and was subsequently studied by other researchers, viz., Iwasawa \cite{iwasawa2001collected}, Schmidt \cite{Schmidt}, Suzuki \cite{suzuki}, etc, and has become an important domain of research in group theory since the last century. The study of sets of equivalence classes of subgroups of finite groups by partially ordering them is a novel technique that is explored in \cite{Mainardis1998}, \cite{tarnauceanu2015posetsclassesisomorphicsubgroups}, etc.\par
	For a positive integer $n$, the \emph{dihedral group} of order
	$2n$, denoted by $D_n$, is defined as
	\begin{equation*}
		D_n=\bigl<r,s\hspace{0.2cm}|\hspace{0.2cm} r^n=e,\hspace{0.2cm}s^2=e,\hspace{0.2cm}srs^{-1}=r^{-1}\bigr>.
	\end{equation*} 
	\par Recently, studies of various combinatorial structures as well as subgroup lattices on $D_n$ is explored. Kharat $et \hspace{0.1cm}al.$ in \cite{kharatarticle} studied the poset of all pronormal and Hall subgroups of $D_n$. In \cite{sachinardra}, Ballal $et \hspace{0.1cm} al.$ studied the intersection hypergraph of $D_n$ and examined several graph theoretic properties viz., diameter, girth and chromatic number. Similar studies on several graphs on $D_n$ and its related properties are explored in \cite{reddy2023} and \cite{articleselva}. However, in this paper we consider equivalence classes of subgroups of $D_n$ and study its structure.  \par
	In the following result, a complete listing of subgroups of $D_n$ is given.
	\begin{theorem}
		\cite{conrad} Every subgroup of $D_n$ is cyclic or dihedral. A complete listing of the subgroups is as follows: 
		\begin{enumerate}
			\item $\bigl<r^d\bigr>$, where $d|n$, with index $2d$,
			\item $\bigl<r^d,r^is\bigr>$, where $d|n$ and $0\le i\le d-1$, with index $d$.
		\end{enumerate}
		Every subgroup of $D_n$ occurs exactly once in this listing.
	\end{theorem}
	\begin{remark}
		\begin{enumerate}
			\item A subgroup of $D_n$ is said to be of \textbf{Type (1)} if it is cyclic as stated in (1) of Theorem 1.1.
			\item A subgroup of $D_n$ is said to be of \textbf{Type (2)} if it is dihedral subgroup as stated in (2) of Theorem 1.1.
		\end{enumerate}
	\end{remark}
	We denote the lattice of subgroups of a group $G$ by $\mathcal{L}(G)$ and the identity element of $G$ by $e$. For a subgroup $H$ of a finite group, we denote the set $\{o(x)\hspace{0.1cm}|\hspace{0.1cm}x\in H\}$ by $\pi_e(H)$, where $o(x)$ denotes the order of $x$. The \emph{exponent} of a finite group $G$ is the smallest positive integer $n$, such that $g^n=e$, for all $g\in G$. We denote the \emph{chain} with $n$ elements by $C_n$ and the lattice of positive divisors of an integer $n$ by $T(n)$. For an odd prime $p$, upto isomorphism, there is a unique non abelian group of order $p^3$ with exponent $p$. This group is isomorphic to the group of all upper unitriangular $3\times 3$ matrices over $\mathbb{Z}_p$ and is denoted by $\text{Heis}(\mathbb{Z}_p)$. Unless stated otherwise, $p_1, p_2, \dots, p_k$ will always denote distinct odd primes.
	\par In this paper we study the poset $\widetilde\Pi(G)$ of classes of subgroups with same set of orders of elements. We characterize all finite groups $G$ for which $\widetilde\Pi(G)$ is a chain and in particular, $C_2$. We prove that the poset $\widetilde\Pi(D_n)$ is a lattice and characterize all values of $n$ for which $\widetilde\Pi(D_n)$ is distributive and modular lattice. \par
	For more details on lattices, groups and subgroup lattices, one may refer \cite{Gratzer}, \cite{robinson1996course} and \cite{Schmidt}. 
	\section{The Poset $\mathbf{\widetilde\Pi(\emph{G})}$}
		For a finite group $G$, define an equivalence relation $\equiv$ on $\mathcal{L}(G)$ as follows:
	\begin{equation*}
		H_1 \equiv H_2 \hspace{0.2cm} \text{if and only if} \hspace{0.2cm} \pi_e(H_1)=\pi_e(H_2),      
	\end{equation*} Let $\quotientthree{\mathcal{L}(G)}{\equiv}$ denotes the set of equivalence classes with respect to $\equiv$.
	In \cite{tarnauceanu2015posetsclassesisomorphicsubgroups}, Tărnăuceanu raised a problem to study the posets of classes of subgroups of finite groups with same set of orders of elements. Motivated by this question, we have introduced the following partial ordering on $\quotientthree{\mathcal{L}(G)}{\equiv}$ 
		\begin{equation*}
		[H_1]\lesssim [H_2] \hspace{0.2cm}\text{if and only if } \hspace{0.2cm} \pi_e(H_1)\subseteq \pi_e(H_2), 
	\end{equation*}
	 We call the poset $\biggl(\quotientthree{\mathcal{L}(G)}{\equiv}, \lesssim \biggr)$, as \emph{the poset of classes of subgroups of $G$ with same set of orders of elements} and denote it by ${\widetilde\Pi({G})}$. In this section, we study the poset $\widetilde \Pi(G)$ in particular for dihedral groups. \vspace{0.2cm}\\
		In the following result, we have characterized all finite groups $G$, for which $\widetilde \Pi(G)$ are chains. 
	\begin{theorem}
		The poset $\widetilde\Pi(G)$ is a chain if and only if $G$ is a $p$-group.
	\end{theorem}
	\begin{proof}
		Let $\widetilde\Pi(G)$ be a chain and on the contrary, assume that $G$ is not a $p$-group, then for distinct prime factors $p_1$ and $p_2$ of $|G|$, there exists elements $a$ and $b$ in $G$ of order $p_1$ and $p_2$, respectively. Clearly $\pi_e(\bigl<a\bigr>)=\{1,p_1\}$ and $\pi_e(\bigl<b\bigr>)=\{1,p_2\}$ and consequently, $[\bigl<a\bigr>]$ and $[\bigl<b\bigr>]$ are incomparable in $\widetilde\Pi(G)$, a contradiction. So, $G$ is a $p$-group.\par
		Conversely, let $G$ be a $p$-group of order $p^n$ and $A$ and $B$ be two subgroups of $G$ with $\text{exp}(A)=p^{m_1}$ and $\text{exp}(B)=p^{m_2}$. Without the loss of generality, if $m_1\le m_2$, then $\pi_e(A)=T(p^{m_1})\subseteq T(p^{m_2})=\pi_e(B)$, which implies that $[A]$ and $[B]$ are comparable and consequently $\widetilde\Pi(G)$ is a chain.
	\end{proof}
	\begin{corollary}
		Let $G_1$ and $G_2$ be such that $\widetilde\Pi(G_1)\cong \widetilde\Pi(G_2)$. If $G_1$ is a $p$-group, then so is $G_2$. Moreover, $\text{exp}(G_1)=\text{exp}(G_2)$. If $\text{exp}(G_1)=p^n$, then $\widetilde\Pi(G_1)\cong C_{n+1} \cong \widetilde\Pi(G_2) $.
	\end{corollary}
	In the next result, we have characterized classes of finite groups for which when $\widetilde\Pi(G)$ is isomorphic to $C_2$.
	\begin{theorem}
		For a finite group $G$, the poset $\widetilde\Pi(G)\cong C_2$ if and only if $G$ is a cyclic group of order $p$ or $G$ is an elementary abelian $p$-group or $G$ has a subgroup isomorphic to $\text{Heis}(\mathbb{Z}_p)$ with $\text{exp}(G)=p$, where $p$ is a prime.  
	\end{theorem}
	\begin{proof}
		Certainly if $G\cong \mathbb{Z}_p$ or $G\cong \mathbb{Z}_p\times \dots \times \mathbb{Z}_p$ or $G$ has a subgroup isomorphic to $\text{Heis}(\mathbb{Z}_p)$ with $\text{exp}(G)=p$, then $\widetilde\Pi(G)\cong C_2$. \par
		Conversely, let $\widetilde\Pi(G)\cong C_2$, then $\text{exp}(G)=p$, else there is a subgroup of $G$ of order $p^2$ which implies that there is a sublattice of $\widetilde\Pi(G)$ which is isomorphic to $ C_3$, which is not possible. Now, if $G$ is abelian, then $G$ is a cyclic group of order $p$ or an elementary abelian $p$-group. If $G$ is non-abelian, then $G$ has a non abelian subgroup $H$ of order $p^3$. Since $\exp(G)=p$, so $\exp(H)=p$ and hence $H\cong \text{Heis}(\mathbb{Z}_p)$.  
	\end{proof}
	We now focus on finite cyclic groups. The next result shows that our poset under study coincides with subgroup lattice in the case of $\mathbb{Z}_n$
	\begin{theorem}
		For all positive integer $n$, the poset ${\widetilde\Pi({\mathbb{Z}_n})}$ is isomorphic to $\mathcal{L}(\mathbb{Z}_n)$ and in particular, ${\widetilde\Pi({\mathbb{Z}_n})}$ is a lattice.
	\end{theorem}
	\begin{proof}
		Consider the map $\Phi: \mathcal{L}(\mathbb{Z}_n)\to \widetilde\Pi(\mathbb{Z}_n)$ by $\Phi(H)=[H]$. It is clear that $|{\widetilde\Pi({\mathbb{Z}_n})}| \le |\mathcal{L}(\mathbb{Z}_n)|$. Now, let $H$ and $K$ be two distinct subgroups of $\mathbb{Z}_n$, with order $m_1$ and $m_2$, respectively. Certainly, when $m_1> m_2$, $\pi_e(H)\ne \pi_e(K)$ as $m_1\notin \pi_e(K)$ but $m_1\in \pi_e(H)$ and therefore, $[H]$ and $[K]$ are distinct classes. Thus, $|\mathcal{L}(\mathbb{Z}_n)|\le|{\widetilde\Pi({\mathbb{Z}_n})}|$ and hence, $|{\widetilde\Pi({\mathbb{Z}_n})}| = |\mathcal{L}(\mathbb{Z}_n)|$. This implies that $\Phi$ is well defined and bijective. \par 
		If $[H_1]\lesssim [H_2]$, then $\pi_e(H_1)\subseteq\pi_e(H_2)$. Let $x\in H_1$ be arbitrary, then $o(x)\in \pi_e(H_1)$, and as $\pi_e(H_1)\subseteq\pi_e(H_2)$ so, there is $y\in H_2$ with $o(y)=o(x)$. As, subgroups of order $o(x)$ is unique, we have $<y>=<x>$, therefore there is a positive integer $m'$ with $x=y^{m'}$ and thus $x\in H_2$, since $x$ is arbitrary, we have $H_1\le H_2$. Moreover, if $H_1\le H_2$, then certainly, $\pi_e(H_1)\subseteq \pi_e(H_2)$ and thus $[H_1]\lesssim [H_2]$. So, $\Phi$ is a bijective isotone between $\mathcal{L}(\mathbb{Z}_n)$ and $\widetilde\Pi(\mathbb{Z}_n)$, which implies that $\Phi$ is a lattice isomorphism.
	\end{proof}
	In order to show that $\widetilde{\Pi}(D_n)$ is a lattice, we essentially need the following result.
	\begin{lemma}
		Let $G$ be a finite group and $[H_1],[H_2]\in \widetilde\Pi(G)$. Consider the set \[X=\bigcap_{\mathclap{{\substack{\bar L \in \mathcal{L}(G) \\ \pi_e(H),\pi_e(K)\subseteq \pi_e(\bar L)}}}} \pi_e(\bar L).\]
		The set  $\{[H],[K]\}$ has a least upper bound if and only if there exists a subgroup $L'\in \mathcal{L}(G)$ with $\pi_e(L')=X$. Moreover, $[H]\vee' [K]=[L']$. \par Similarly, consider the set \[Y=\bigcup_{\mathclap{{\substack{\widehat L \in \mathcal{L}(G) \\ \pi_e(\widehat L)\subseteq\pi_e(H),\pi_e(K)}}}} \pi_e(\widehat L).\] The set  $\{[H],[K]\}$ has a greatest lower bound if and only if there exists a subgroup $L''\in \mathcal{L}(G)$ with $\pi_e(L'')=Y$. Moreover, $[H]\wedge' [K]=[L'']$.
	\end{lemma}
	\begin{proof}
		Let there exists a subgroup $L'\in \mathcal{L}(G)$ with  $\pi_e(L')=X$. Since $\pi_e(H),\pi_e(K)\subseteq X$, it follows that $[L']$ is an upper bound of $\{[H],[K]\}\}$. Moreover, if $[S]$ is an upper bound of $\{[H],[K]\}\}$, then $\pi_e(H),\pi_e(K)\subseteq \pi_e(S)$ and thus $\pi_e(L')\subseteq X\subseteq \pi_e(S)$ and therefore $[H]\vee'[K]=[L']$. \par Conversely, suppose $[H]\vee'[K]=[L']$, then by definition of $X$ it follows that $X\subseteq \pi_e(L')$. Now suppose $\bar L$ is such that $\pi_e(H), \pi_e(K)\subseteq \bar L$, then $[\bar L]$ is an upper bound of $[H]$ and$[K]$. Since $[L']$ is the least upper bound of $[H]$ and $[K]$, it follows that $\pi_e(L')\subseteq X$. The second part follows on similar lines.
	\end{proof}
	In the light of Lemma 2.4, we prove the following:
	\begin{theorem}
		For positive integer $n$, the poset $\widetilde\Pi(D_n)$ is a lattice.
	\end{theorem}
	\begin{proof}
		Let $n=2^\alpha \prod\limits_{i=1}^{k}p_i^{t_i}$, where $p_i$'s are distinct odd primes with $\alpha, t_i's\ge 0$ and $[H],[K]\in \widetilde \Pi(D_n)$ be arbitrary classes. It is sufficient to show that the least upper bound and the greatest lower bound of $\{[H],[K]\}$ exists. So consider the following cases: \vspace{0.2cm}\\
		\textbf{Case 1:} If $2\notin \pi_e(H),\pi_e(K)$, then $H=\biggl<r^{2^\alpha\prod\limits_{i=1}^{k} p_{i}^{u_i}}\biggr>$ and $K=\biggl<r^{2^\alpha\prod\limits_{i=1}^{k} p_{i}^{v_i}}\biggr>$, where $0\le u_i,v_i\le t_i$ and $1\le i\le k$, and therefore, \begin{equation*}
			\pi_e(H)= T \biggl( \prod\limits_{i=1}^{k} p_i^{t_i-u_i} \biggr) \hspace{0.2cm}\text{and}\hspace{0.2cm}     \pi_e(K)= T \biggl(\prod\limits_{i=1}^{k} p_i^{t_i-v_i} \biggr).
		\end{equation*} 
		So, we claim that $[L']$ is the least upper bound of $\{[H],[K]\}$, where $L'=\biggl<r^{2^\alpha\prod\limits_{i=1}^{k} p_{i}^{\min\{u_i,v_i\}}}\biggr>$. Clearly, $\pi_e(H),\pi_e(K)\subseteq \pi_e(L')$ as $H,K \le L'$ and therefore, \begin{equation*}
			\bigcap_{\mathclap{{\substack{\bar L \in \mathcal{L}(G) \\ \pi_e(H),\pi_e(K)\subseteq \pi_e(\bar L)}}}} \pi_e(\bar L) \subseteq \pi_e(L'). 
		\end{equation*}
		Now, suppose that $\prod\limits_{i=1}^{k} p_{i}^{x_i}\in \pi_e(L')$ with $0\le x_i\le t_i-\min \{u_i,v_i\}$, $1\le i\le k$ and $\bar L\le G$ with $\pi_e(H),\pi_e(K)\subseteq\pi_e(\bar L)$. Then $\bar L$ has an element of order $\prod\limits_{i=1}^{k} p_{i}^{t_i-u_i}$, say, $\biggl<r^{a_12^\alpha\prod\limits_{i=1}^{k} p_{i}^{u_i}}\biggr>$ and an element of order $\prod\limits_{i=1}^{k} p_{i}^{t_i-v_i}$, say, $\biggl<r^{a_22^\alpha\prod\limits_{i=1}^{k} p_{i}^{v_i}}\biggr>$. Now, as $\biggl<r^{a_12^\alpha\prod\limits_{i=1}^{k} p_{i}^{u_i}}\biggr>=\biggl<r^{2^\alpha\prod\limits_{i=1}^{k} p_{i}^{u_i}}\biggr>$ and $\biggl<r^{a_22^\alpha\prod\limits_{i=1}^{k} p_{i}^{v_i}}\biggr>=\biggl<r^{2^\alpha\prod\limits_{i=1}^{k} p_{i}^{v_i}}\biggr>$, so $\biggl<r^{a_12^\alpha\prod\limits_{i=1}^{k}p_i^{u_i}}\biggr>\vee\biggl<r^{a_22^\alpha\prod\limits_{i=1}^{k}p_i^{v_i}}\biggr>=\biggl<r^{2^\alpha\prod\limits_{i=1}^{k}p_i^{u_i}}\biggr>\vee\biggl<r^{2^\alpha\prod\limits_{i=1}^{k}p_i^{v_i}}\biggr>\le \bar L$ and therefore, $\biggl<r^{2^\alpha\prod\limits_{i=1}^{k}p_i^{\min\{u_i,v_i\}}} \biggr>\le \bar L$, and consequently, $\bar L$ contains an element of order $\prod\limits_{i=1}^{k}p_i^{t_i-\min\{u_i,v_i\}}$. Note that, $\biggl<r^{2^\alpha\prod\limits_{i=1}^{k}p_i^{\min\{u_i,v_i\}}} \biggr>$ contains elements of order $\prod\limits_{i=1}^{k}p_i^{x}$ where $0\le x \le t_i-\min \{u_i,v_i\}$, as cyclic group satisfies converse to Lagrange's theorem, so $\bar L$ contains an element of order $\prod\limits_{i=1}^{k}p_i^{x}$ where $0\le x \le t_i-\min \{u_i,v_i\}$. Since $\bar L$ was arbitrary, we have,
		\begin{equation*}
			\pi_e(L')\subseteq\bigcap_{\mathclap{{\substack{\bar L \in \mathcal{L}(G) \\ \pi_e(H),\pi_e(K)\subseteq \pi_e(\bar L)}}}} \pi_e(\bar L).
		\end{equation*}
		Thus $[L']$ is the least upper bound of $\{[H],[K]\}$. \par Now, we claim that $[L'']$ is the greatest lower bound of $\{[H],[K]\}$, where $L''=\biggl<r^{2^\alpha\prod\limits_{i=1}^{k}p_i^{\max\{u_i,v_i\}}}\biggr>$. Clearly, $\pi_e(L'')\subseteq \pi_e(H),\pi_e(K)$ as $L''\le H, K$ and therefore,
		\begin{equation*}
			\pi_e(L'')\subseteq\bigcup_{\mathclap{{\substack{\widehat L \in \mathcal{L}(G) \\ \pi_e(\widehat L)\subseteq\pi_e(H),\pi_e(K)}}}} \pi_e(\widehat L).
		\end{equation*}
		Let $\widehat L\le G$ with $\pi_e(\widehat L)\subseteq \pi_e(H),\pi_e(K)$. Suppose $m\in \pi_e(\widehat L)$, then $m\in \pi_e(H), \pi_e(K)$ and which implies that $m$ divides $|H|$ and $|K|$, i.e., $m$ divides $\prod\limits_{i=1}^{k}p_i^{t_i-u_i}$ and $\prod\limits_{i=1}^{k}p_i^{t_i-v_i}$. So, $m$ divides $\gcd\biggl(\prod\limits_{i=1}^{k}p_i^{t_i-u_i},\prod\limits_{i=1}^{k}p_i^{t_i-u_i}\biggr) $, i.e., $m$ divides $\prod\limits_{i=1}^{k}p_i^{\min \{t_i-u_i,t_i-v_i\}}=\prod\limits_{i=1}^{k}p_i^{t_i-\max \{u_i,v_i\}}$ and this implies $m$ divides the order of $r^{2^\alpha\prod\limits_{i=1}^{k}p_i^{\max\{u_i,v_i\}}} $ and hence $m\in \pi_e(L'')$. Since $\widehat L$ was arbitrary, we have,
		\begin{equation*}
			\bigcup_{\mathclap{{\substack{\widehat L \in \mathcal{L}(G) \\ \pi_e(\widehat L)\subseteq\pi_e(H),\pi_e(K)}}}} \pi_e(\widehat L)\subseteq \pi_e(L'').
		\end{equation*}
		Thus, $[L'']$ is the greatest lower bound of $\{[H],[K]\}$.\vspace{0.2cm}\\
		\textbf{Case 2:} If $2\notin \pi_e(H)$ and $2\in \pi_e(K)$, then we have the following subcases: \vspace{0.2cm}\\
		\textbf{Subcase 2.1:} If $\{1,2\}=\pi_e(K)$, then $K=\bigl<r^ls\bigr>$ or $\biggl<r^{2^{\alpha-1}\prod\limits_{i=1}^{k}p_i^{t_i}}\biggr>$ or $\biggl<r^{2^{\alpha-1}\prod\limits_{i=1}^{k}p_i^{t_i}},r^js\biggr> $, where $0\le l\le n-1$, $0\le j\le \biggl(2^{\alpha-1} \prod\limits_{i=1}^{k}p_i^{t_i}\biggr)-1$. As, $2\notin \pi_e(H)$, we have $H=\biggl<r^{2^\alpha\prod\limits_{i=1}^{k} p_{i}^{u_i}}\biggr>$ with $0\le u_i\le t_i$, $1\le i\le k$. We claim that $[L']$ is the least upper bound of $\{[H],[K]\}$, where $L'=\biggl<r^{{2^\alpha} \prod\limits_{i=1}^{k}p_i^{u_i}},s\biggr>$ . Certainly, $\pi_e(H),\pi_e(K)\subseteq \pi_e(L')$ as $H\le L'$ and $\pi_e(K)=\{1,2\}\subseteq\pi_e(L')$ and thus, we have, \begin{equation*}
			\bigcap_{\mathclap{{\substack{\bar L \in \mathcal{L}(G) \\ \pi_e(H),\pi_e(K)\subseteq \pi_e(\bar L)}}}} \pi_e(\bar L) \subseteq \pi_e(L'). 
		\end{equation*}
		Let $m\in \pi_e(L')$ and $\bar L\le G$ with $\pi_e(H),\pi_e(K)\subseteq \pi_e(\bar L)$, then $m=2$ or $m\in T\biggl(\prod\limits_{i=1}^{k}p_i^{t_i-u_i}\biggr)$. As, $\pi_e(K)=\{1,2\}$ and $\pi_e(H)=T\biggl(\prod\limits_{i=1}^{k}p_i^{t_i-u_i}\biggr)$, so $m\in \pi_e(\bar L)$. Since, $\bar L$ is arbitrary, we have,    \begin{equation*}
			\pi_e(L')\subseteq\bigcap_{\mathclap{{\substack{\bar L \in \mathcal{L}(G) \\ \pi_e(H),\pi_e(K)\subseteq \pi_e(\bar L)}}}} \pi_e(\bar L).
		\end{equation*}
		Therefore, $[L']$ is the least upper bound of $\{[H],[K]\}$. \par
		We claim that $[L'']$ is the greatest lower bound of $\{[H],[K]\}$, where $L''=\bigl<e\bigr>$. Certainly, $\pi_e(L'')\subseteq \pi_e(H), \pi_e(K)$, as $\pi_e(L'')=\{1\}$ and therefore we have,   \begin{equation*}
			\pi_e(L'')\subseteq\bigcup_{\mathclap{{\substack{\widehat L \in \mathcal{L}(G) \\ \pi_e(\widehat L)\subseteq\pi_e(H),\pi_e(K)}}}} \pi_e(\widehat L).
		\end{equation*}
		Let $m\in \pi_e(\widehat L)$, where $\widehat L\le G$ with $\pi_e(\widehat L)\subseteq \pi_e(H),\pi_e(K)$ so that $m\in T\biggl(\prod\limits_{i=1}^{k}p_i^{t_i-u_i}\biggr)\bigcap \{1,2\}=\{1\} $, which implies that $m=1\in \pi_e(L'') $, and hence,   \begin{equation*}
			\bigcup_{\mathclap{{\substack{\widehat L \in \mathcal{L}(G) \\ \pi_e(\widehat L)\subseteq\pi_e(H),\pi_e(K)}}}} \pi_e(\widehat L)\subseteq \pi_e(L'').
		\end{equation*}
		Therefore, $[L'']$ is the greatest lower bound of $\{[H],[K]\}$.\vspace{0.2cm}\\
		\textbf{Subcase 2.2:} If $\{1,2\}\subsetneq \pi_e(K)$, then $K=\biggl<r^{2^\beta\prod\limits_{i=1}^{k}p_i^{v_i}}\biggr>$, where $0\le\beta\le \alpha-1$ and if $\beta=\alpha-1$, then not all $v_i=t_i$ or $K=\biggl<r^{2^\beta\prod\limits_{i=1}^{k}p_i^{v_i}}, r^js\biggr>$, where $0\le \beta \le \alpha $, $0\le v_i\le t_i$, $1\le i\le k$, $0\le j\le \biggl(2^\beta\prod\limits_{i=1}^{k}p_i^{v_i}\biggr)-1$ and if $\beta$ is $\alpha-1$ or $\alpha $, then not all $v_i=t_i$. As, $2\notin \pi_e(H)$, we have $H=\biggl<r^{2^\alpha\prod\limits_{i=1}^{k} p_{i}^{u_i}}\biggr>$ with $0\le u_i\le t_i$, $1\le i\le k$. We claim that $[L']$ is the least upper bound of $\{[H],[K]\}$, where $L'=\biggl<r^{2^\beta \prod\limits_{i=1}^{k}\min\{u_i,v_i\}},s\biggr>$. Certainly, $\pi_e(H), \pi_e(K)\subseteq\pi_e(L')$ as $H\le L'$ and $\pi_e(K)=T\biggl(2^{\alpha-\beta}\prod\limits_{i=1}^{k}p_i^{t_i-v_i}\biggr)\subseteq T\biggl(2^{\alpha-\beta}\prod\limits_{i=1}^{k}p_i^{t_i-\min \{u_i,v_i\}}\biggr)=\pi_e(L') $. So, we have,  \begin{equation*}
			\bigcap_{\mathclap{{\substack{\bar L \in \mathcal{L}(G) \\ \pi_e(H),\pi_e(K)\subseteq \pi_e(\bar L)}}}} \pi_e(\bar L) \subseteq \pi_e(L'). 
		\end{equation*}  
		For $m\in \pi_e(L')$, we have $m\in T\biggl(2^{\alpha-\beta} \prod\limits_{i=1}^{k}p_i^{t_i-\min\{u_i,v_i\}}\biggr)\bigcup \{2\}$. Let $\bar L\le G$ with $\pi_e(H),\pi_e(K)\subseteq \pi_e(\bar L)$. Then $\bar L$ contains an element of order $\prod\limits_{i=1}^{k}p_i^{t_i-u_i}$, since all elements of order $\prod\limits_{i=1}^{k}p_i^{t_i-u_i}$ are in $H$ and generates $H$, hence, $H\le L$. Moreover, $L$ contains an element of order $2^{\alpha-\beta}\prod\limits_{i=1}^{k}p_i^{t_i-v_i}$. We have that all element of order $2^{\alpha-\beta}\prod\limits_{i=1}^{k}p_i^{t_i-v_i}$ are in $\biggl<r^{2^\beta \prod\limits_{i=1}^{k}p_i^{v_i}}\biggr>$ or $\biggl<r^{2^\beta \prod\limits_{i=1}^{k}p_i^{v_i}},r^js\biggr>$, and all such elements generate $\biggl<r^{2^\beta \prod\limits_{i=1}^{k}p_i^{v_i}}\biggr>$. Hence,$\biggl<r^{2^\beta \prod\limits_{i=1}^{k}p_i^{v_i}}\biggr>\le \bar L$ and therefore, $\biggl<r^{2^\alpha\prod\limits_{i=1}^{k} p_{i}^{u_i}}\biggr>\vee\biggl<r^{2^\beta \prod\limits_{i=1}^{k}p_i^{v_i}}\biggr>=\biggl<r^{2^\beta \prod\limits_{i=1}^{k}p_i^{\min\{u_i,v_i\}}}\biggr>\le \bar L$. So,  $T\biggl(2^{\alpha-\beta} \prod\limits_{i=1}^{k}p_i^{t_i-\min\{u_i,v_i\}}\biggr)\subseteq \pi_e(\bar L)$ and as $2\in \pi_e(K)\subseteq \pi_e(\bar L)$, we have, in particular, $m\in \pi_e(\bar L)$. So, \begin{equation*}
			\pi_e(L')\subseteq\bigcap_{\mathclap{{\substack{\bar L \in \mathcal{L}(G) \\ \pi_e(H),\pi_e(K)\subseteq \pi_e(\bar L)}}}} \pi_e(\bar L).
		\end{equation*}
		Therefore, $[L']$ is the least upper bound of $\{[H],[K]\}$. \par
		Now, we claim that $[L'']$ is the greatest lower bound of $\{[H],[K]\}$, where $L''=\biggl<r^{2^\alpha \prod\limits_{i=1}^{k}p_i^{\max\{u_i,v_i\}}}\biggr>$. Certainly, $\pi_e(L'')\subseteq \pi_e(H),\pi_e(K)$, as, $L''\le H$ and $\pi_e(L'')=T\biggl(\prod\limits_{i=1}^{k}p_i^{t_i-\max\{u_i,v_i\}}\biggr)\subseteq T\biggl(2^{\alpha-\beta} \prod\limits_{i=1}^{k}p_i^{t_i-v_i}\biggr)\bigcup \{2\}=\pi_e(K)$ and thus, we have,   
		\begin{equation*}
			\pi_e(L'')\subseteq\bigcup_{\mathclap{{\substack{\widehat L \in \mathcal{L}(G) \\ \pi_e(\widehat L)\subseteq\pi_e(H),\pi_e(K)}}}} \pi_e(\widehat L).
		\end{equation*}
		Let $\widehat L\le G$ with $\pi_e(\widehat L)\subseteq \pi_e(H),\pi_e(K)$ and consider, $m\in \pi_e(L)$, then $m\in T\biggl(\prod\limits_{i=1}^{k}p_i^{t_i-{u_i}}\biggr)\bigcap$ $ \biggl\{ T\biggl(2^{\alpha-\beta} \prod\limits_{i=1}^{k}p_i^{t_i-v_i}\biggr)\bigcup \{2\}\biggr\}$ $=T\biggl(\prod\limits_{i=1}^{k}p_i^{t_i-{u_i}}\biggr)\bigcap  T\biggl(2^{\alpha-\beta} \prod\limits_{i=1}^{k}p_i^{t_i-v_i}\biggr)$ and thus $m$ divides $\gcd \biggl( \prod\limits_{i=1}^{k}p_i^{t_i-{u_i}},  2^{\alpha-\beta} \prod\limits_{i=1}^{k}p_i^{t_i-v_i}  \biggr)=\gcd \biggl( \prod\limits_{i=1}^{k}p_i^{t_i-{u_i}},  \prod\limits_{i=1}^{k}p_i^{t_i-v_i}  \biggr)$\\$=\prod\limits_{i=1}^{k}p_i^{\min \{t_i-u_i,t_i-v_i\}}=\prod\limits_{i=1}^{k}p_i^{t_i-\max \{u_i,v_i\}}$. As, $\pi_e(L'')=T\biggl(\prod\limits_{i=1}^{k}p_i^{t_i-\max\{u_i,v_i\}}\biggr)$, we have, in particular, $m\in \pi_e(L'')$. Since $\widehat L$ is arbitrary, we have,  
		\begin{equation*}
			\bigcup_{\mathclap{{\substack{\widehat L \in \mathcal{L}(G) \\ \pi_e(\widehat L)\subseteq\pi_e(H),\pi_e(K)}}}} \pi_e(\widehat L)\subseteq \pi_e(L'').
		\end{equation*}
		Therefore, $[L'']$ is the greatest lower bound of $\{[H],[K]\}$.\vspace{0.2cm}\\
		\textbf{Case 3:} If $2\in \pi_e(H), \pi_e(K)$, then it is sufficient to consider $\{1,2\}\subsetneq\pi_e(H),\pi_e(K)$, as, without the loss of generality, if $\pi_e(H)=\{1,2\}$, then $\pi_e(H)\subseteq\pi(K)$ and thus, $[K]$ is the least upper bound of $\{[H],[K]\}$ and $[H]$ is the greatest lower bound of $\{[H],[K]\}$. So, if $\{1,2\}\subsetneq \pi_e(H),\pi_e(K)$, then $H=\biggl<r^{2^{\beta_1}\prod\limits_{i=1}^{k}p_i^{u_i}}\biggr>$, where $0\le\beta_1\le \alpha-1$ and if $\beta_1=\alpha-1$, then not all $u_i=t_i$ or $H=\biggl<r^{2^{\beta_1}\prod\limits_{i=1}^{k}p_i^{u_i}}, r^{j_1}s\biggr>$, where $0\le \beta_1\le \alpha$ and if $\beta_1$ is $\alpha-1$ or $\alpha$, then not all $u_i=t_i$ and $K=\biggl<r^{2^{\beta_2}\prod\limits_{i=1}^{k}p_i^{v_i}}\biggr>$, where $0\le \beta_2\le \alpha-1$ and if $\beta_2=\alpha-1$, then not all $v_i=t_i$ or $K=\biggl<r^{2^{\beta_2}\prod\limits_{i=1}^{k}p_i^{v_i}}, r^{j_2}s\biggr>$, where, $0\le \beta_2\le \alpha$ and if $\beta_2$ is $\alpha-1$ or $\alpha$, then not all $v_i=w_i$ and $0\le u_i,v_i\le t_i$, $1\le i\le k$, $0\le j_1\le \biggl(2^{\beta_1}\prod\limits_{i=1}^{k}p_i^{u_i}\biggr)-1$, $0\le j_2\le \biggl(2^{\beta_2}\prod\limits_{i=1}^{k}p_i^{v_i}\biggr)-1$. We claim that $[L']$ is the least upper bound of $\{[H],[K]\}$, where $L'=\biggl<r^{2^{\min\{\beta_1,\beta_2\}}\prod\limits_{i=1}^{k}p_i^{\min\{u_i,v_i\}}},s\biggr>$. Certainly, by the choice of $L'$, it is evident that $\pi_e(H),\pi_e(K)\subseteq \pi_e(L')$ and hence, we have,  \begin{equation*}
			\bigcap_{\mathclap{{\substack{\bar L \in \mathcal{L}(G) \\ \pi_e(H),\pi_e(K)\subseteq \pi_e(\bar L)}}}} \pi_e(\bar L) \subseteq \pi_e(L'). 
		\end{equation*} 
		Let $m\in \pi_e(L')$, then $m\in T\biggl(2^{\alpha-\min\{\beta_1,\beta_2\}}\prod\limits_{i=1}^{k}p_i^{t_i-\min\{u_i,v_i\}}\biggr)$. Let $\bar L\le G$ with $\pi_e(H),\pi_e(K)\subseteq\pi_e(\bar L)$, then $\bar L$ contains an element of order $2^{\alpha-\beta_1}\prod\limits_{i=1}^{k}p_i^{t_i-u_i}$, since all elements of order $2^{\alpha-\beta_1}\prod\limits_{i=1}^{k}p_i^{t_i-u_i}$ are in $\biggl<r^{2^{\beta_1}\prod\limits_{i=1}^{k}p_i^{u_i}}\biggr>$ or $\biggl<r^{2^{\beta_1}\prod\limits_{i=1}^{k}p_i^{u_i}},r^js\biggr>$ and generates $\biggl<r^{2^{\beta_1}\prod\limits_{i=1}^{k}p_i^{u_i}}\biggr>$, hence, $\biggl<r^{2^{\beta_1}\prod\limits_{i=1}^{k}p_i^{u_i}}\biggr> \le \bar L$. Moreover, $\bar L$ contains an element of order $2^{\alpha-\beta_2}\prod\limits_{i=1}^{k}p_i^{t_i-v_i}$. We have that all element of order $2^{\alpha-\beta_2}\prod\limits_{i=1}^{k}p_i^{t_i-v_i}$ are in $\biggl<r^{2^{\beta_2} \prod\limits_{i=1}^{k}p_i^{v_i}}\biggr>$ or $\biggl<r^{2^{\beta_2} \prod\limits_{i=1}^{k}p_i^{v_i}},r^js\biggr>$ , and all such elements generate $\biggl<r^{2^{\beta_2} \prod\limits_{i=1}^{k}p_i^{v_i}}\biggr>$. Hence,$\biggl<r^{2^{\beta_2} \prod\limits_{i=1}^{k}p_i^{v_i}}\biggr>\le \bar L$ and therefore, $\biggl<r^{2^{\beta_1}\prod\limits_{i=1}^{k} p_{i}^{u_i}}\biggr>\vee\biggl<r^{2^{\beta_2} \prod\limits_{i=1}^{k}p_i^{v_i}}\biggr>=\biggl<r^{2^{\min\{\beta_1,\beta_2\}} \prod\limits_{i=1}^{k}p_i^{\min\{u_i,v_i\}}}\biggr>\le \bar L$. So, $T\biggl(2^{\alpha-\min\{\beta_1,\beta_2\}} \prod\limits_{i=1}^{k}p_i^{t_i-\min\{u_i,v_i\}}\biggr)$\\$\subseteq \pi_e(\bar L)$ and as $2\in \pi_e(H)\subseteq \pi_e(\bar L)$, we have, in particular, $m\in \pi_e(\bar L)$. So, \begin{equation*}
			\pi_e(L')\subseteq\bigcap_{\mathclap{{\substack{\bar L \in \mathcal{L}(G) \\ \pi_e(H),\pi_e(K)\subseteq \pi_e(\bar L)}}}} \pi_e(\bar L).
		\end{equation*}
		Therefore, $[L']$ is the least upper bound of $\{[H],[K]\}$. \par 
		We claim that $[L'']$ is the greatest lower bound of $\{[H],[K]\}$, where $L''=\biggl<r^{2^{\max\{\beta_1,\beta_2\}}\prod\limits_{i=1}^{k}p_i^{\max\{u_i,v_i\}}}\biggr>$. Certainly, $\pi_e(L'')\subseteq \pi_e(H),\pi_e(K)$ as $\pi_e(L'')=T\biggl(2^{\alpha-\max\{\beta_1,\beta_2\}}\prod\limits_{i=1}^{k}p_i^{t_i-\max\{u_i,v_i\}}\biggr)\bigcup \{2\}\subseteq T\biggl(2^{\alpha-\beta_1}\prod\limits_{i=1}^{k}p_i^{t_i-u_i}\biggr)=\pi_e(H)$ and $\pi_e(L'')=T\biggl(2^{\alpha-\max\{\beta_1,\beta_2\}}\prod\limits_{i=1}^{k}p_i^{t_i-\max\{u_i,v_i\}}\biggr)\bigcup \{2\}\subseteq T\biggl(2^{\alpha-\beta_2}\prod\limits_{i=1}^{k}p_i^{t_i-v_i}\biggr)=\pi_e(K)$ and hence, we have, 
		\begin{equation*}
			\pi_e(L'')\subseteq\bigcup_{\mathclap{{\substack{\widehat L \in \mathcal{L}(G) \\ \pi_e(\widehat L)\subseteq\pi_e(H),\pi_e(K)}}}} \pi_e(\widehat L).
		\end{equation*}
		Let $\widehat L\le G$ with $\pi_e(\widehat L)\subseteq \pi_e(H),\pi_e(K)$ and consider, $m\in \pi_e(\widehat L)$, then $m\in T\biggl(2^{\alpha-\beta_1}\prod\limits_{i=1}^{k}p_i^{t_i-{u_i}}\biggr)\bigcap $\\$  T\biggl(2^{\alpha-{\beta_2}} \prod\limits_{i=1}^{k}p_i^{t_i-v_i}\biggr)$ and thus $m$ divides $\gcd \biggl(2^{\alpha-{\beta_1}} \prod\limits_{i=1}^{k}p_i^{t_i-{u_i}},  2^{\alpha-\beta_2} \prod\limits_{i=1}^{k}p_i^{t_i-v_i}  \biggr)=$\\$2^{\min\{\alpha-\beta_1,\alpha-\beta_2\}}\prod\limits_{i=1}^{k}p_i^{\min \{t_i-u_i,t_i-v_i\}}=2^{\alpha-\max\{\beta_1,\beta_2\}}\prod\limits_{i=1}^{k}p_i^{t_i-\max \{u_i,v_i\}}$. As, $\pi_e(L'')=$\\$T\biggl(2^{\alpha-\max\{\beta_1,\beta_2\}}\prod\limits_{i=1}^{k}p_i^{t_i-\max\{u_i,v_i\}}\biggr)$, we have, in particular, $m\in \pi_e(L'')$. Since $\widehat L$ is arbitrary, we have,  
		\begin{equation*}
			\bigcup_{\mathclap{{\substack{\widehat L \in \mathcal{L}(G) \\ \pi_e(\widehat L)\subseteq\pi_e(H),\pi_e(K)}}}} \pi_e(\widehat L)\subseteq \pi_e(L'').
		\end{equation*}
		Therefore, $[L'']$ is the greatest lower bound of $\{[H],[K]\}$.
	\end{proof}
	\textbf{Examples:}\begin{enumerate}
		\item $\widetilde\Pi(D_{2^\alpha})\cong C_{\alpha+1}$, a chain of length $\alpha+1$.
		\item For distinct odd primes $p_1,p_2$, $\widetilde\Pi(D_{p_1p_2})\cong \mathcal{P}(\{1,2,3\})$, the power set of $\{1,2,3\}$.
		\item For any odd prime $p$, $\widetilde\Pi(D_p)\cong M_2$ and $\widetilde\Pi(D_{p^\alpha})\cong T(p_1^\alpha p_2)$, where $p_1,p_2$ are any distinct primes.
	\end{enumerate}
	The following Theorem gurantees that for $n=\prod\limits_{i=1}^{k}p_i^{t_i}$, where $p_i$'s are distinct odd primes and $t_i's$ are non negative, $\widetilde \Pi (D_n)$ is a distributive lattice.
	\begin{theorem}
		Let $n=\prod\limits_{i=1}^{k}p_i^{t_i}$, where $p_i$'s are distinct odd primes and $t_i's\ge 0$, then $\widetilde \Pi(D_n)\cong T(n) \times C_2 $.
	\end{theorem}
	\begin{proof}
		We first note that, $|T(n)|=\prod\limits_{i=1}^{k}(t_i+1)$ and therefore, $|T(n)\times C_2|=2\prod\limits_{i=1}^{k}(t_i+1)$. As each Type (1) subgroups belongs to distinct classes in $\widetilde\Pi(D_n)$ and $2\notin \pi_e(H)$, for all Type (1) subgroup $H$, so the number of distinct classes containing Type (2) subgroups is same as the number of distinct Type (1) subgroups and hence, $|\widetilde \Pi(D_n)|=2|\text{Number of Type (1) subgroups}|=2|T(n)|=2\prod\limits_{i=1}^{k}(t_i+1)$. Therefore, $|\widetilde \Pi(D_n)|=|T(n)\times C_2|$. \par Consider the map $\phi: \widetilde \Pi(D_n)\to T(n)\times C_2$, given by, 
		\begin{equation*}
			\biggl[\biggl<r^{\prod\limits_{i=1}^{k}p_i^{u_i}}\biggr>\biggr] \longmapsto\ \prod\limits_{i=1}^{k}p_i^{t_i-u_i} \hspace{0.2cm}\text{and \hspace{0.2cm}} \biggl[\biggl<r^{\prod\limits_{i=1}^{k}p_i^{u_i}},s\biggr>\biggr] \longmapsto\ 2\prod\limits_{i=1}^{k}p_i^{t_i-u_i}.
		\end{equation*}
		If $\phi\biggl(  \biggl[\biggl<r^{\prod\limits_{i=1}^{k}p_i^{u_i}}\biggr>\biggr]\biggr) =\phi\biggl(  \biggl[\biggl<r^{\prod\limits_{i=1}^{k}p_i^{v_i}}\biggr>\biggr]\biggr) $, then by definition of $\phi$, $\prod\limits_{i=1}^{k}p_i^{t_i-u_i}=\prod\limits_{i=1}^{k}p_i^{t_i-v_i}$, which implies that $\prod\limits_{i=1}^{k}p_i^{u_i-v_i}=1$, so $u_i=v_i$, for all $i$, $1\le i\le k$ and therefore, $\biggl[\biggl<r^{\prod\limits_{i=1}^{k}p_i^{u_i}}\biggr>\biggr]=\biggl[\biggl<r^{\prod\limits_{i=1}^{k}p_i^{v_i}}\biggr>\biggr]$. Similarly, if $\phi\biggl(  \biggl[\biggl<r^{\prod\limits_{i=1}^{k}p_i^{u_i}},s\biggr>\biggr]\biggr) =\phi\biggl(  \biggl[\biggl<r^{\prod\limits_{i=1}^{k}p_i^{v_i}},s\biggr>\biggr]\biggr) $, then by definition of $\phi$, $2\prod\limits_{i=1}^{k}p_i^{t_i-u_i}=2\prod\limits_{i=1}^{k}p_i^{t_i-v_i}$, which implies that $\prod\limits_{i=1}^{k}p_i^{u_i-v_i}=1$, so $u_i=v_i$, for all $i$, $1\le i\le k$ and therefore, $\biggl[\biggl<r^{\prod\limits_{i=1}^{k}p_i^{u_i}},s\biggr>\biggr]=\biggl[\biggl<r^{\prod\limits_{i=1}^{k}p_i^{v_i}},s\biggr>\biggr]$, which implies that $\phi$ is an injective map. Since, $|\widetilde\Pi(D_n)|=|T(n)\times C_2|$, we have, $\phi$ is a bijective map.\par
		We now show that $\phi$ is a lattice homomorphism, i.e., $\phi([H]\vee'[K])=\text{lcm}(\phi([H]),\phi([K]))$ and $\phi([H]\wedge'[K])=\gcd(\phi([H]),\phi([K]))$, for all $[H],[K]\in \widetilde \Pi(D_n)$. So, consider the following cases:\vspace{0.2cm}\\
		\textbf{Case 1:} If both $H$ and $K$ are of Type (1) subgroups, then 
		$[H]=\biggl[\biggl< r^{\prod\limits_{i=1}^{k}p_i^{u_i}}\biggr>\biggr]$ and $[K]=\biggl[\biggl< r^{\prod\limits_{i=1}^{k}p_i^{v_i}}\biggr>\biggr]$, $0\le u_i,v_i\le t_i$, $1\le i\le k$. By Theorem 2.5 and by the definition of $\phi$, we have, $\phi([H]\vee'[K])=\phi\biggl(\biggl[\biggl<r^{\prod\limits_{i=1}^{k}p_i^{\min \{u_i,v_i\}}}\biggr>\biggr]\biggr)=\prod\limits_{i=1}^{k}p_i^{t_i-\min\{u_i,v_i\}}=\prod\limits_{i=1}^{k}p_i^{\max\{t_i-u_i,t_i-v_i\}}=\text{lcm}\biggl(\prod\limits_{i=1}^{k}p_i^{t_i-u_i},\prod\limits_{i=1}^{k}p_i^{t_i-v_i}\biggr)=\text{lcm}(\phi([H]),\phi([K])).$ \par Similarly, $\phi([H]\wedge'[K])=\phi\biggl(\biggl[\biggl<r^{\prod\limits_{i=1}^{k}p_i^{\max \{u_i,v_i\}}}\biggr>\biggr]\biggr)=\prod\limits_{i=1}^{k}p_i^{t_i-\max\{u_i,v_i\}}=\prod\limits_{i=1}^{k}p_i^{\min\{t_i-u_i,t_i-v_i\}}=\text{gcd}\biggl(\prod\limits_{i=1}^{k}p_i^{t_i-u_i},\prod\limits_{i=1}^{k}p_i^{t_i-v_i}\biggr)=\text{gcd}(\phi([H]),\phi([K])).$ \vspace{0.2cm}\\
		\textbf{Case 2:} If $H$ is of Type (1) and $K$ is of Type (2), then $[H]=\biggl[\biggl< r^{\prod\limits_{i=1}^{k}p_i^{u_i}}\biggr>\biggr]$ and $[K]=\biggl[\biggl< r^{\prod\limits_{i=1}^{k}p_i^{v_i}},s\biggr>\biggr]$, $0\le u_i,v_i\le t_i$, $1\le i\le k$. By Theorem 2.5 and by the definition of $\phi$, we have, $\phi([H]\vee'[K])=\phi\biggl(\biggl[\biggl<r^{\prod\limits_{i=1}^{k}p_i^{\min \{u_i,v_i\}}},s\biggr>\biggr]\biggr)=2\prod\limits_{i=1}^{k}p_i^{t_i-\min\{u_i,v_i\}}=2\prod\limits_{i=1}^{k}p_i^{\max\{t_i-u_i,t_i-v_i\}}=\text{lcm}\biggl(\prod\limits_{i=1}^{k}p_i^{t_i-u_i},2\prod\limits_{i=1}^{k}p_i^{t_i-v_i}\biggr)=\text{lcm}(\phi([H]),\phi([K])).$ \par Similarly, $\phi([H]\wedge'[K])=\phi\biggl(\biggl[\biggl<r^{\prod\limits_{i=1}^{k}p_i^{\max \{u_i,v_i\}}}\biggr>\biggr]\biggr)=\prod\limits_{i=1}^{k}p_i^{t_i-\max\{u_i,v_i\}}=\prod\limits_{i=1}^{k}p_i^{\min\{t_i-u_i,t_i-v_i\}}=\text{gcd}\biggl(\prod\limits_{i=1}^{k}p_i^{t_i-u_i},2\prod\limits_{i=1}^{k}p_i^{t_i-v_i}\biggr)=\text{gcd}(\phi([H]),\phi([K])).$ \vspace{0.2cm}\\
		\textbf{Case 3:} If both $H$ and $K$ are of Type (2) subgroups, then 
		$[H]=\biggl[\biggl< r^{\prod\limits_{i=1}^{k}p_i^{u_i}},s\biggr>\biggr]$ and $[K]=\biggl[\biggl< r^{\prod\limits_{i=1}^{k}p_i^{v_i}},s\biggr>\biggr]$, $0\le u_i,v_i\le t_i$, $1\le i\le k$. By Theorem 2.5 and by the definition of $\phi$, we have, $\phi([H]\vee'[K])=\phi\biggl(\biggl[\biggl<r^{\prod\limits_{i=1}^{k}p_i^{\min \{u_i,v_i\}}},s\biggr>\biggr]\biggr)=2\prod\limits_{i=1}^{k}p_i^{t_i-\min\{u_i,v_i\}}=2\prod\limits_{i=1}^{k}p_i^{\max\{t_i-u_i,t_i-v_i\}}=\text{lcm}\biggl(2\prod\limits_{i=1}^{k}p_i^{t_i-u_i},2\prod\limits_{i=1}^{k}p_i^{t_i-v_i}\biggr)=\text{lcm}(\phi([H]),\phi([K])).$ \par Similarly, $\phi([H]\wedge'[K])=\phi\biggl(\biggl[\biggl<r^{\prod\limits_{i=1}^{k}p_i^{\max \{u_i,v_i\}}},s\biggr>\biggr]\biggr)=2\prod\limits_{i=1}^{k}p_i^{t_i-\max\{u_i,v_i\}}=2\prod\limits_{i=1}^{k}p_i^{\min\{t_i-u_i,t_i-v_i\}}=\text{gcd}\biggl(2\prod\limits_{i=1}^{k}p_i^{t_i-u_i},2\prod\limits_{i=1}^{k}p_i^{t_i-v_i}\biggr)=\text{gcd}(\phi([H]),\phi([K])).$
	\end{proof}
	In order to characterize all values of $n$ for which $\widetilde{\Pi}(G)$ is a distributive lattice, the following characterization theorem due to Birkhoff for distributive lattices is needed.
	\begin{theorem} (\cite{birkhoff1940lattice})
		A lattice is distributive if and only if it does not contain a sublattice isomorphic to a pentagon lattice $(N_5)$ or a diamond lattice $(M_3)$.
	\end{theorem}
\begin{figure}[h]
	\centering
	\begin{tikzpicture}[
		every node/.style={circle,fill=black,inner sep=0pt,minimum size=6pt},
		]
		
		\begin{scope}[xshift=0cm]
			
			\node (zero) at (0,0) {};
			\node (a)    at (-1,1) {};
			\node (b)    at (-1,2) {};
			\node (c)    at (1.2,1.5) {};
			\node (one)  at (0,3) {};
			
			\draw (zero) -- (a) -- (b) -- (one);
			\draw (zero) -- (c) -- (one);
			
			\node[draw=none, fill=none, font=\large] at (0,-1.2) {$N_5$};
			
		\end{scope}
		
		\begin{scope}[xshift=7cm]
			
			\node (zero2) at (0,0) {};
			\node (x)     at (-1.5,1.5) {};
			\node (y)     at (0,1.5) {};
			\node (z)     at (1.5,1.5) {};
			\node (one2)  at (0,3) {};
			
			\draw (zero2) -- (x) -- (one2);
			\draw (zero2) -- (y) -- (one2);
			\draw (zero2) -- (z) -- (one2);
			
			\node[draw=none, fill=none, font=\large] at (0,-1.2) {$M_3$};
			
		\end{scope}
		
	\end{tikzpicture}
	
	\caption{}
	\label{fig:lattices}
	
\end{figure}
	\begin{theorem}
		For all $n$, the lattice $\widetilde\Pi(D_n)$ does not contain a sublattice isomorphic to $M_3$.  
	\end{theorem}
	\begin{proof}
		For $n=1$, $\widetilde\Pi(D_n)\cong C_2$ and for $n=2^k$, $\widetilde\Pi(D_n)\cong C_{k+1}$, so the result holds in these cases. So, suppose $n=2^{\alpha}\prod\limits_{i=1}^{k}p_i^{t_i}$, where $p_i$'s are distinct odd primes, $\alpha$ is a non-negative integer and not all $t_i$'s are zero. On the contrary, if there exists a sublattice isomorphic to $M_3$, then there exist distinct $[H],[K],[L],[H]\vee'[K],[H]\wedge'[K]$ with $[H]\vee'[K]=[H]\vee'[L]=[K]\vee'[L]$ and $[H]\wedge'[K]=[H]\wedge'[L]=[K]\wedge'[L]$.
	%
				

		Now, consider the following cases:\\
		\textbf{Case 1:} If $2\notin \pi_e(H), \pi_e(K)$, then $H=\biggl<r^{2^\alpha\prod\limits_{i=1}^{k} p_{i}^{u_i}}\biggr>$ and $K=\biggl<r^{2^\alpha\prod\limits_{i=1}^{k} p_{i}^{v_i}}\biggr>$, where $0\le u_i,v_i\le t_i$ and $1\le i\le k$. By Theorem 2.5, $[H]\vee'[K]=\biggl[\biggl<r^{2^\alpha\prod\limits_{i=1}^{k} p_{i}^{\min\{u_i,v_i\}}}\biggr>\biggr]$ and $[H]\wedge'[K]=\biggl[\biggl<r^{2^\alpha\prod\limits_{i=1}^{k} p_{i}^{\max\{u_i,v_i\}}}\biggr>\biggr]$. As, $[H]\vee'[K]=[H]\vee'[L]$ and $2\notin \pi_e\biggl(\biggl<r^{2^\alpha\prod\limits_{i=1}^{k} p_{i}^{\min\{u_i,v_i\}}}\biggr>\biggr) $, therefore $2\notin \pi_e(L)$. So, $L=\biggl<r^{2^\alpha\prod\limits_{i=1}^{k} p_{i}^{w_i}}\biggr>$, for $0\le w_i\le t_i$, $1\le i\le k$. By Theorem 2.5, $[H]\vee'[L]=\biggl[\biggl<r^{2^\alpha\prod\limits_{i=1}^{k} p_{i}^{\min\{u_i,w_i\}}}\biggr>\biggr]$ and $[H]\wedge'[L]=\biggl[\biggl<r^{2^\alpha\prod\limits_{i=1}^{k} p_{i}^{\max\{u_i,w_i\}}}\biggr>\biggr]$. As $[H]\vee'[K]=[H]\vee'[L]$, so $\biggl[\biggl<r^{2^\alpha\prod\limits_{i=1}^{k} p_{i}^{\min\{u_i,v_i\}}}\biggr>\biggr]=\biggl[\biggl<r^{2^\alpha\prod\limits_{i=1}^{k} p_{i}^{\min\{u_i,w_i\}}}\biggr>\biggr]$ and hence, $\pi_e\biggl(\biggl<r^{2^\alpha\prod\limits_{i=1}^{k} p_{i}^{\min\{u_i,v_i\}}}\biggr>\biggr)=\pi_e\biggl(\biggl<r^{2^\alpha\prod\limits_{i=1}^{k} p_{i}^{\min\{u_i,w_i\}}}\biggr>\biggr)$, i.e., $T\biggl(\prod\limits_{i=1}^{k} {p_{i}^{t_i-\min\{u_i,v_i\}}}\biggr)=T\biggl({\prod\limits_{i=1}^{k} p_{i}^{t_i-\min\{u_i,w_i\}}}\biggr)$. Thus, in particular, $\prod\limits_{i=1}^{k} {p_{i}^{t_i-\min\{u_i,v_i\}}}$ divides $\prod\limits_{i=1}^{k} {p_{i}^{t_i-\min\{u_i,w_i\}}}$ and $\prod\limits_{i=1}^{k} {p_{i}^{t_i-\min\{u_i,w_i\}}}$ divides $\prod\limits_{i=1}^{k} {p_{i}^{t_i-\min\{u_i,v_i\}}}$ and therefore, $\prod\limits_{i=1}^{k} {p_{i}^{t_i-\min\{u_i,v_i\}}}=\prod\limits_{i=1}^{k} {p_{i}^{t_i-\min\{u_i,w_i\}}}$, and consequently, $\min \{u_i,v_i\}=\min \{u_i,w_i\}$, for all $1\le i\le k$. Similarly, as $[H]\wedge'[K]=[H]\wedge'[L]$, so $\biggl[\biggl<r^{2^\alpha\prod\limits_{i=1}^{k} p_{i}^{\max\{u_i,v_i\}}}\biggr>\biggr]=\biggl[\biggl<r^{2^\alpha\prod\limits_{i=1}^{k} p_{i}^{\max\{u_i,w_i\}}}\biggr>\biggr]$, and hence, $\pi_e \biggl(\biggl<r^{2^\alpha\prod\limits_{i=1}^{k} p_{i}^{\max\{u_i,v_i\}}}\biggr>\biggr)= \pi_e \biggl(\biggl<r^{2^\alpha\prod\limits_{i=1}^{k} p_{i}^{\max\{u_i,w_i\}}}\biggr>\biggr)$, i.e., $T\biggl(\prod\limits_{i=1}^{k} {p_{i}^{t_i-\max\{u_i,v_i\}}}\biggr)=T\biggl({\prod\limits_{i=1}^{k} p_{i}^{t_i-\max\{u_i,w_i\}}}\biggr)$. Thus, in particular, $\prod\limits_{i=1}^{k} {p_{i}^{t_i-\max\{u_i,v_i\}}}$ divides $\prod\limits_{i=1}^{k} {p_{i}^{t_i-\max\{u_i,w_i\}}}$ and $\prod\limits_{i=1}^{k} {p_{i}^{t_i-\max\{u_i,w_i\}}}$ divides $\prod\limits_{i=1}^{k} {p_{i}^{t_i-\max\{u_i,v_i\}}}$ and therefore, $\prod\limits_{i=1}^{k} {p_{i}^{t_i-\max\{u_i,v_i\}}}=\prod\limits_{i=1}^{k} {p_{i}^{t_i-\max\{u_i,w_i\}}}$, and consequently, $\max \{u_i,v_i\}=\max \{u_i,w_i\}$, for all $1\le i\le k$. Therefore, $v_i=w_i$ for all $1\le i\le k$ and hence $[K]=[L]$, a contradiction.\vspace{0.1cm}\\
		\textbf{Case 2:} If $2 \notin \pi_e(H)$, but $2\in \pi_e(K)$, then clearly, $H=\biggl<r^{2^\alpha\prod\limits_{i=1}^{k} p_{i}^{u_i}}\biggr>$ for $0\le u_i\le t_i, 1\le i\le k$.\\
		\textbf{Subcase 2.1:} If $\pi_e(K)=\{1,2\}$, then by Theorem 2.5, $[H]\vee'[K]=\biggl[\biggl<r^{2^\alpha\prod\limits_{i=1}^{k} p_{i}^{u_i}},s\biggr>\biggr]$ and $[H]\wedge'[K]=[\bigl<e\bigr>]$. Certainly, $2\in \pi_e(L)$, else if $2\notin \pi_e(L)$, then by Theorem 2.5, $2\notin \pi_e(\widetilde T)$, for all $\widetilde T\in [H]\vee'[K]$, which is absurd as $2\in \pi_e\biggl( \biggl<r^{2^\alpha\prod\limits_{i=1}^{k} p_{i}^{u_i}},s\biggr>\biggr)$. So, $\pi_e(K)\subseteq \pi_e(L)$ and hence $[K]\lesssim [L]$, a contradiction.\\
		\textbf{Subcase 2.2:} If $\{1,2\}\subsetneq \pi_e(K)$, then $[K]$ is $\biggl[\biggl<r^{2^\beta\prod\limits_{i=1}^{k}p_i^{v_i}},s\biggr>\biggr]$, $0\le \beta\le \alpha$ and if $\beta$ is $\alpha-1$ or $\alpha$, then not all $v_i=t_i$. By Theorem 2.5, $[H]\vee'[K]=\biggl[\biggl<r^{2^\beta\prod\limits_{i=1}^{k}p_i^{\min\{u_i,v_i\}}},s\biggr>\biggr]$ and $[H]\wedge'[K]=\biggl[\biggl<r^{2^\alpha\prod\limits_{i=1}^{k}p_i^{\max\{u_i,v_i\}}}\biggr>\biggr]$. Now, if $2\in \pi_e(L)$, then $\pi_e(\bigl<s\bigr>)\subseteq \pi_e(K), \pi_e(L)$ and hence $[\bigl<s\bigr>]\lesssim [K]\wedge [L]$, so $2\in \widetilde T$ for all $\widetilde T\in [K]\wedge'[L]=[H]\wedge'[K]$, but this is not possible as $2\notin \pi_e\biggl(\biggl<r^{2^\alpha\prod\limits_{i=1}^{k}p_i^{\max\{u_i,v_i\}}}\biggr>\biggr)$. So, $2\notin \pi_e(L)$ and thus $L=\biggl<r^{2^\alpha\prod\limits_{i=1}^{k} p_{i}^{w_i}}\biggr>$, so by Theorem 2.5, $2\notin \widetilde T$ for any $\widetilde T\in [H]\vee'[L]=[H]\vee'[K]$, which is a contradiction.\vspace{0.2cm}\\
		\textbf{Case 3:} If $2\in \pi_e(H),\pi_e(K)$, then $\{1,2\}\subsetneq \pi_e(H),\pi_e(K)$, else if $\pi_e(H)$ or $\pi_e(K)$ is $\{1,2\}$, then $[H]$ and $[K]$ are comparable. So, $[H]$ is $\biggl[\biggl<r^{2^{\beta_1}\prod\limits_{i=1}^{k}p_i^{u_i}},s\biggr>\biggr]$, and $[K]$ is $\biggl[\biggl<r^{2^{\beta_2}\prod\limits_{i=1}^{k}p_i^{v_i}},s\biggr>\biggr]$, $0\le \beta_1,\beta_2\le \alpha$ and if $\beta_1, \beta_2$ are $\alpha-1$ or $\alpha$, then not all $u_i$'s and $v_i$'s are $t_i$. By Theorem 2.5, $[H]\vee'[K]=\biggl[\biggl<r^{2^{\min\{\beta_1,\beta_2\}}\prod\limits_{i=1}^{k}p_i^{\min\{u_i,v_i\}}},s\biggr>\biggr]$ and $[H]\wedge'[K]=\biggl[\biggl<r^{2^{\max\{\beta_1,\beta_2\}}\prod\limits_{i=1}^{k}p_i^{\max\{u_i,v_i\}}},s\biggr>\biggr]$, so certainly $\{1,2\}\subsetneq \pi_e(L) $ and hence $[L]=\biggl[\biggl<r^{2^{\beta_3}\prod\limits_{i=1}^{k}p_i^{w_i}},s\biggr>\biggr]$ and by Theorem 2.5, $[H]\vee'[L]=\biggl[\biggl<r^{2^{\min\{\beta_1,\beta_3\}}\prod\limits_{i=1}^{k}p_i^{\min\{u_i,w_i\}}},s\biggr>\biggr]$ and $[H]\wedge'[L]=\biggl[\biggl<r^{2^{\max\{\beta_1,\beta_3\}}\prod\limits_{i=1}^{k}p_i^{\max\{u_i,w_i\}}},s\biggr>\biggr]$. As, $[H]\vee'[K]=[H]\vee'[L]$, we have, $\biggl[\biggl<r^{2^{\min\{\beta_1,\beta_2\}}\prod\limits_{i=1}^{k}p_i^{\min\{u_i,v_i\}}},s\biggr>\biggr]=\biggl[\biggl<r^{2^{\min\{\beta_1,\beta_3\}}\prod\limits_{i=1}^{k}p_i^{\min\{u_i,w_i\}}},s\biggr>\biggr]$, which implies that $T\biggl(2^{\alpha-\min\{\beta_1,\beta_2\}}\prod\limits_{i=1}^{k}p_i^{t_i-\min\{u_i,v_i\}}\biggr)\cup\{2\}$\\$=T\biggl(2^{\alpha-\min\{\beta_1,\beta_3\}}\prod\limits_{i=1}^{k}p_i^{t_i-\min\{u_i,w_i\}}\biggr)\cup\{2\}$. So, $2^{\alpha-\min\{\beta_1,\beta_2\}}\prod\limits_{i=1}^{k}p_i^{t_i-\min\{u_i,v_i\}}$ divides $2^{\alpha-\min\{\beta_1,\beta_3\}}\prod\limits_{i=1}^{k}p_i^{t_i-\min\{u_i,w_i\}}$ and vice versa, therefore, $2^{\alpha-\min\{\beta_1,\beta_2\}}\prod\limits_{i=1}^{k}p_i^{t_i-\min\{u_i,v_i\}}$=$2^{\alpha-\min\{\beta_1,\beta_3\}}\prod\limits_{i=1}^{k}p_i^{t_i-\min\{u_i,w_i\}}$, which implies that $\min\{\beta_1,\beta_2\}=\min\{\beta_1,\beta_3\}$ and $\min\{u_i,v_i\}=\min\{u_i,w_i\}$, for $1\le i\le k$. Similarly, on comparing $[H]\wedge'[K]=[H]\wedge'[L]$, we get,  $\max\{\beta_1,\beta_2\}=\max\{\beta_1,\beta_3\}$ and $\max\{u_i,v_i\}=\max\{u_i,w_i\}$, for $1\le i\le k$, so $\beta_2=\beta_3$ and $v_i=w_i$, for all $i$, and hence, $[K]=[L]$, a contradiction.
	\end{proof}
	In the next result, we show that for $n=2\prod\limits_{i=1}^{k}p_i^{t_i}$, the lattice $\widetilde\Pi(D_n)$ contains a sublattice isomorphic to $N_5$, except for $k=1$. Moreover, we obtain the complete structure of the sublattice isomorphic to $N_5$ in $\widetilde\Pi(D_n)$, for $n=2\prod\limits_{i=1}^{k}p_i^{t_i}$ with $k\ge 2$. 
	\begin{theorem}
		Let $n=2^\alpha\prod\limits_{i=1}^{k}p_i^{t_i}$, with atleast one $t_i\ne 0$. The lattice $\widetilde\Pi(D_n)$ contains a sublattice isomorphic to $N_5$ if and only if $n=2^\alpha \prod\limits_{i=1}^{k}p_i^{t_i}$ with $\alpha \ge 2$ or $n=2\prod\limits_{i=1}^{k}p_i^{t_i}$ with $k\ge 2$. 
	\end{theorem}
	\begin{proof}
		Let $n=2^\alpha \prod\limits_{i=1}^{k}p_i^{t_i}$, with $\alpha \ge 2 $ and not all $t_i$'s are zero. Without the loss of generality, assume that $t_1\ne 0$. Consider the sets $A_1=\{1,2\}$, $A_2=\{1,2,p_1\}$ ,$A_3=\{1,2,4\}$, $A_4=\{1,2,p_1,2p_1\}$, $A_5=\{1,2,4,p_1,2p_1,4p_1\}$ and the subgroups $H_1=\bigl<s\bigr>$, $H_2=\bigl<r^{2^\alpha p_1^{t_1-1}p_2^{t_2}\dots p_k^{t_k}},s\bigr>$, $H_3=\bigl<r^{2^{\alpha-2}p_1^{t_1}p_2^{t_2}\dots p_k^{t_k}},s\bigr>$, $H_4=\bigl<r^{2^{\alpha-1}p_1^{t_1-1}p_2^{t_2}\dots p_k^{t_k}},s\bigr>$, $H_5=\bigl<r^{2^{\alpha-2}p_1^{t_1-1}p_2^{t_2}\dots p_k^{t_k}},s\bigr>$ of $D_n$. Clearly $\pi_e(H_i)=A_i$, for $1\le i\le 5$. Moreover, $[H_1],[H_2],[H_3],[H_4]$ and $[H_5]$ forms the vertices of a pentagon as depicted in Figure 2.
		\begin{figure}[h!]
		\centering
		\begin{tikzpicture}[scale=1.5,  every node/.style={circle, fill=black, minimum size=6pt, inner sep=0pt}]
		
		 Nodes in set U
					\node[label=left:${[H_4]}$] (u2) at (0,1) {};
					\node[label=left:${[H_2]}$] (u1) at (0,0.5) {};
					\node[label=right:${[H_5]}$] (u5) at (0.5,1.5) {};
		           \node[label=right:${[H_3]}$] (u3) at (1,0.75) {};
				\node[label=right:${[H_1]}$] (u4) at (0.5,0) {};
		
	 Edges
				\draw (u1) -- (u2);
				\draw (u2) -- (u5);
		       \draw (u1) -- (u4);
		     \draw (u4) -- (u3);
		     \draw (u3) -- (u5);

		\end{tikzpicture}
		\caption{}
		\end{figure}\\    
		Now, suppose that for $n=2\prod\limits_{i=1}^{k}p_i^{t_i}$, the lattice $\widetilde\Pi(D_n)$ contains a sublattice isomorphic to $N_5$, then there exist distinct classes $[H],[K],[L],[H]\vee'[K]$ and $[H]\wedge'[K]$ with $[H]\vee'[K]=[L]\vee'[K]$ and $[H]\wedge'[K]=[L]\wedge'[K]$. 
	%
				
		Now consider the following cases:\\
		\textbf{Case 1:} If $2\notin \pi_e(H),\pi_e(K)$, then $H=\biggl<r^{2\prod\limits_{i=1}^{k}p_i^{u_i}}\biggr>$ and $K=\biggl<r^{2\prod\limits_{i=1}^{k}p_i^{v_i}}\biggr>$. By, Theorem 2.5, $[H]\vee'[K]=\biggl[\biggl<r^{2\prod\limits_{i=1}^{k}p_i^{\min\{u_i,v_i\}}}\biggr>\biggr]$ and $[H]\wedge'[K]=\biggl[\biggl<r^{2\prod\limits_{i=1}^{k}p_i^{\max\{u_i,v_i\}}}\biggr>\biggr]$, which implies that $2\notin \pi_e(L)$, so $L=\biggl<r^{2\prod\limits_{i=1}^{k}p_i^{w_i}}\biggr>$. Using Theorem 2.5 and the fact $[H]\vee' [K]=[L]\vee'[K]$, we have $\biggl[\biggl<r^{2\prod\limits_{i=1}^{k}p_i^{\min\{u_i,v_i\}}}\biggr>\biggr]=\biggl[\biggl<r^{2\prod\limits_{i=1}^{k}p_i^{\min\{w_i,v_i\}}}\biggr>\biggr]$, so $\pi_e\biggl(\biggl<r^{2\prod\limits_{i=1}^{k}p_i^{\min\{u_i,v_i\}}}\biggr>\biggr)=\pi_e\biggl(\biggl<r^{2\prod\limits_{i=1}^{k}p_i^{\min\{w_i,v_i\}}}\biggr>\biggr)$ which implies that $T\biggl(\prod\limits_{i=1}^{k}p_i^{t_i-\min\{u_i,v_i\}}\biggr)=T\biggl(\prod\limits_{i=1}^{k}p_i^{t_i-\min\{w_i,v_i\}}\biggr)$. As, $\prod\limits_{i=1}^{k}p_i^{t_i-\min\{u_i,v_i\}}$ divides $\prod\limits_{i=1}^{k}p_i^{t_i-\min\{w_i,v_i\}}$ and $\prod\limits_{i=1}^{k}p_i^{t_i-\min\{w_i,v_i\}}$ divides $\prod\limits_{i=1}^{k}p_i^{t_i-\min\{u_i,v_i\}}$, we have $\prod\limits_{i=1}^{k}p_i^{t_i-\min\{u_i,v_i\}}=\prod\limits_{i=1}^{k}p_i^{t_i-\min\{w_i,v_i\}}$. So, $\min\{u_i,v_i\}=\min\{w_i,v_i\}$, for all $1\le i\le k$. \par
		Similarly, using Theorem 2.5 and the fact $[H]\wedge'[K]=[L]\wedge'[K]$, we have $\max\{u_i,v_i\}=\max\{w_i,v_i\}$, for all $1\le i\le k$, which implies that $u_i=w_i$, for all $1\le i\le k$ and consequently, $H=L$, a contradiction.\\
		\textbf{Case 2:} If $2\notin \pi_e(H)$ but $2\in \pi_e(K)$, then $H=\biggl<r^{2\prod\limits_{i=1}^{k}p_i^{u_i}}\biggr>$. Consider the following subcases.\\
		\textbf{Subcase 2.1:} If $\pi_e(K)=\{1,2\}$, then $[K]=[\bigl<s\bigr>]$ and by Theorem 2.5, $[H]\vee'[K]=\biggl[\biggl<r^{2\prod\limits_{i=1}^{k}p_i^{u_i}},s\biggr>\biggr]$ and $[H]\wedge'[K]=[\bigl<e\bigr>]$. Certainly $2\notin \pi_e(L)$, else if $2\in \pi_e(L)$, then $[K]\lesssim [L]$, which is not possible. So, $L=\biggl<r^{2\prod\limits_{i=1}^{k}p_i^{w_i}}\biggr>$. Using Theorem 2.5 and the fact $[H]\vee'[K]=[L]\vee'[K]$, we have, $u_i=w_i$, for all $1\le i\le k$ and consequently, $H=L$, a contradiction.\\
		\textbf{Subcase 2.2:} If $\{1,2\}\subsetneq\pi_e(K)$, then $[K]=\biggl[\biggl<r^{2\prod\limits_{i=1}^{k}p_i^{v_i}},s\biggl>\biggl]$ or $\biggl[\biggl<r^{\prod\limits_{i=1}^{k}p_i^{v_i}},s\biggl>\biggl]$, with not all $v_i$'s are $t_i$.
		For both the choices of $[K]$, using Theorem 2.5 and the fact $[H]\vee'[K]=[L]\vee'[K]$ and $[H]\wedge'[K]=[L]\wedge [K]$, we have $H=L$, which is a contradiction.
		 \vspace{0.2cm}\\
		\textbf{Case 3:} If $2\in \pi_e(H)$ but $2\notin \pi_e(K)$, then $K=\biggl<r^{2\prod\limits_{i=1}^{k}p_i^{v_i}}\biggl>$. Consider the following subcases:\\
		\textbf{Subcase 3.1:} If $\pi_e(H)=\{1,2\}$, then $[H]=[\bigl<s\bigr>]$ and by Theorem 2.5, $[H]\vee'[K]=\biggl[\biggl<r^{2\prod\limits_{i=1}^{k}p_i^{v_i}},s\biggr>\biggr]$ and $[H]\wedge'[K]=[\bigl<e\bigr>]$. Certainly, $2\in \pi_e(L)$ and as $[L]$ is distinct from $[H]$, so, $\{1,2\}\subsetneq \pi_e(L)$. Thus, $[L]=\biggl[\biggl<r^{2\prod\limits_{i=1}^{k}p_i^{w_i}},s\biggr> \biggr]$ or $\biggl[\biggl<r^{\prod\limits_{i=1}^{k}p_i^{w_i}},s\biggr> \biggr]$, with not all $w_i=t_i$. If $[L]=\biggl[\biggl<r^{2\prod\limits_{i=1}^{k}p_i^{w_i}},s\biggr> \biggr]$, then by Theorem 2.5 and by the fact $[H]\vee'[K]=[L]\vee'[K]$, we have $\biggl[\biggl<r^{2\prod\limits_{i=1}^{k}p_i^{v_i}},s\biggr>\biggr]=\biggl[\biggl<r^{2\prod\limits_{i=1}^{k}p_i^{\min\{w_i,v_i\}}},s\biggr>\biggr]$, i.e., $T\biggl(\prod\limits_{i=1}^{k}p_i^{t_i-v_i}\biggr)\cup\{2\}=T\biggl(\prod\limits_{i=1}^{k}p_i^{t_i-\min \{w_i,v_i\}}\biggr)\cup\{2\}$ and in particular, $T\biggl(\prod\limits_{i=1}^{k}p_i^{t_i-v_i}\biggr)=T\biggl(\prod\limits_{i=1}^{k}p_i^{t_i-\min \{w_i,v_i\}}\biggr)$, which implies that $\min\{w_i,v_i\}=v_i$, for all $1\le i\le k$. Similarly, using Theorem 2.5 and the fact $[H]\wedge'[K]=[L]\wedge'[K]$, we have $[\bigl<e\bigr>]=\biggl[\biggl<r^{2\prod\limits_{i=1}^{k}p_i^{\max\{w_i,v_i\}}}\biggr>\biggr]=\biggl[\biggl<r^{2\prod\limits_{i=1}^{k}p_i^{w_i}}\biggr>\biggr]$, hence $w_i=t_i$ for all $1\le i\le k$ and consequently, $[L]=[\bigl<s\bigr>]=[H]$, a contradiction.\par
		If $[L]=\biggl[\biggl<r^{\prod\limits_{i=1}^{k}p_i^{w_i}},s\biggr> \biggr]$, then by Theorem 2.5 and as $[H]\vee'[K]=[L]\vee'[K]$, we have $\biggl[\biggl<r^{2\prod\limits_{i=1}^{k}p_i^{v_i}},s\biggr> \biggr]=\biggl[\biggl<r^{\prod\limits_{i=1}^{k}p_i^{\min\{w_i,v_i\}}},s\biggr> \biggr]$, i.e., $T\biggl(\prod\limits_{i=1}^{k}p_i^{t_i-v_i}\biggr)\cup\{2\}=T\biggl(2\prod\limits_{i=1}^{k}p_i^{t_i-\min\{w_i,v_i\}}\biggr)$. Certainly, there exists $i_0$ with $\min\{w_{i_0},v_{i_0}\}\neq t_{i_0}$, else $[H]\vee'[K]=[H]$, which is not possible. Then $2p_{i_0}\in T\biggl(2\prod\limits_{i=1}^{k}p_i^{t_i-\min\{w_i,v_i\}}\biggr)$ but $2p_{i_0} \notin T\biggl(\prod\limits_{i=1}^{k}p_i^{t_i-v_i}\biggr)\cup\{2\}$, a contradiction.\\
		\textbf{Subcase 3.2:} If $\{1,2\}\subsetneq\pi_e(H)$, then $[H]=\biggl[\biggl<r^{2\prod\limits_{i=1}^{k}p_i^{u_i}},s\biggr>\biggr]$ or $\biggl[\biggl<r^{\prod\limits_{i=1}^{k}p_i^{u_i}},s\biggr>\biggr]$, with not all $u_i=t_i$.\par Suppose that $[H]=\biggl[\biggl<r^{2\prod\limits_{i=1}^{k}p_i^{u_i}},s\biggr>\biggr]$, then using Theorem 2.5, we have $[H]\vee'[K]=\biggl[\biggl<r^{2\prod\limits_{i=1}^{k}p_i^{\min\{u_i,v_i\}}},s\biggr>\biggr]$ and $[H]\wedge'[K]=\biggl[\biggl<r^{2\prod\limits_{i=1}^{k}p_i^{\max\{u_i,v_i\}}}\biggr>\biggr]$. This implies $[L]=\biggl[\biggl<r^{2\prod\limits_{i=1}^{k}p_i^{w_i}},s\biggr>\biggr]$ or $\biggl[\biggl<r^{\prod\limits_{i=1}^{k}p_i^{w_i}},s\biggr>\biggr]$, where not all $w_i=t_i$. \par
		If $[L]=\biggl[\biggl<r^{2\prod\limits_{i=1}^{k}p_i^{w_i}},s\biggr>\biggr]$, then using Theorem 2.5 and the fact $[H]\vee'[K]=[L]\vee'[K]$ and $[H]\wedge'[K]=[L]\wedge'[K]$, it is easy to show that $[H]=[L]$, a contradiction. \par
		 Now, if $[L]=\biggl[\biggl<r^{\prod\limits_{i=1}^{k}p_i^{w_i}},s\biggr>\biggr]$, then using Theorem 2.5 and the fact $[H]\vee'[K]=[L]\vee'[K]$, we have, $\biggl[\biggl<r^{2\prod\limits_{i=1}^{k}p_i^{\min\{u_i,v_i\}}},s\biggr>\biggr]=\biggl[\biggl<r^{\prod\limits_{i=1}^{k}p_i^{\min\{w_i,v_i\}}},s\biggr>\biggr]$, i.e., $T\biggl(\prod\limits_{i=1}^{k}p_i^{t_i-\min\{u_i,v_i\}}\biggr)\cup\{2\}=T\biggl(2\prod\limits_{i=1}^{k}p_i^{t_i-\min\{w_i,v_i\}}\biggr)$, this implies $t_i=\min\{u_i,v_i\}=\min\{w_i,v_i\}$, for $1\le i\le k$. Similarly, using $[H]\wedge'[K]=[L]\wedge'[K]$, we get $\max\{u_i,v_i\}=\max\{w_i,v_i\}$, for $1\le i\le k$. So, $u_i=w_i$, for $1\le i\le k$. Now the join implies $T\biggl(\prod\limits_{i=1}^{k}p_i^{t_i-\min\{u_i,v_i\}}\biggr)\cup\{2\}=T\biggl(2\prod\limits_{i=1}^{k}p_i^{t_i-\min\{u_i,v_i\}}\biggr)$, which holds only if $t_i=\min\{u_i,v_i\}$, for all $i$ and consequently, $[H]\vee'[K]=[\bigl<s\bigr>]$, a contradiction.   \par
		If $[H]=\biggl[\biggl<r^{\prod\limits_{i=1}^{k}p_i^{u_i}},s\biggr>\biggr]$, then clearly $[L]=\biggl[\biggl<r^{\prod\limits_{i=1}^{k}p_i^{w_i}},s\biggr>\biggr]$ and using similar argument as in above paragraph, we have $[H]=[L]$, a contradiction. \vspace{0.2cm}\\
		\textbf{Case 4:} If $2\in \pi_e(H),\pi_e(K)$, then $2\in \pi_e(L)$. As $[H]$ and $[K]$ are incomparable, so, $\{1,2\}\subsetneq [H],[K]$. Thus the choice of $[H]$ are $\biggl[\biggl<r^{2\prod\limits_{i=1}^{k}p_i^{u_i}},s\biggr>\biggr]$ or $\biggl[\biggl<r^{\prod\limits_{i=1}^{k}p_i^{u_i}},s\biggr>\biggr]$ and the choices of $[K]$ are $\biggl[\biggl<r^{2\prod\limits_{i=1}^{k}p_i^{v_i}},s\biggr>\biggr]$ or $\biggl[\biggl<r^{\prod\limits_{i=1}^{k}p_i^{v_i}},s\biggr>\biggr]$, where not all $u_i=t_i$ and not all $v_i=t_i$.\\
		\textbf{Subcase 4.1:} If $[H]=\biggl[\biggl<r^{2\prod\limits_{i=1}^{k}p_i^{u_i}},s\biggr>\biggr]$ and $[K]=\biggl[\biggl<r^{2\prod\limits_{i=1}^{k}p_i^{v_i}},s\biggr>\biggr]$, then the choices of $[L]$ are $\biggl[\biggl<r^{2\prod\limits_{i=1}^{k}p_i^{w_i}},s\biggr>\biggr]$ or $\biggl[\biggl<r^{\prod\limits_{i=1}^{k}p_i^{w_i}},s\biggr>\biggr]$. Suppose $[L]=\biggl[\biggl<r^{2\prod\limits_{i=1}^{k}p_i^{w_i}},s\biggr>\biggr]$, then by Theorem 2.5 and the fact $[H]\vee'[K]=[L]\vee'[K]$ and $[H]\wedge'[K]=[L]\wedge'[K]$, we have $w_i=u_i$ and consequently, $[H]=[L]$, a contradiction. \par
		Now, suppose $[L]=\biggl[\biggl<r^{\prod\limits_{i=1}^{k}p_i^{w_i}},s\biggr>\biggr]$, then using Theorem 2.5 and the fact $[H]\vee'[K]=[L]\vee'[K]$, we have, $\biggl[\biggl<r^{2\prod\limits_{i=1}^{k}p_i^{\min\{u_i,v_i\}}},s\biggr>\biggr]=\biggl[\biggl<r^{\prod\limits_{i=1}^{k}p_i^{\min\{w_i,v_i\}}},s\biggr>\biggr]$, i.e., $T\biggl(\prod\limits_{i=1}^{k}p_i^{t_i-\min\{u_i,v_i\}}\biggr)\cup\{2\}=T\biggl(2\prod\limits_{i=1}^{k}p_i^{t_i-\min\{w_i,v_i\}}\biggr)$, which implies $t_i=\min\{u_i,v_i\}=\min\{w_i,v_i\}$, for all $i$. Similarly, comparing the meet, we get $\max\{u_i,v_i\}=\max\{w_i,v_i\}$, for all $i$. Therefore, $u_i=w_i$, for all $i$ and hence $[H]\vee'[K]=[\bigl<s\bigr>]$, a contradiction.\\
		\textbf{Subcase 4.2:} If $[H]=\biggl[\biggl<r^{2\prod\limits_{i=1}^{k}p_i^{u_i}},s\biggr>\biggr]$ and $[K]=\biggl[\biggl<r^{\prod\limits_{i=1}^{k}p_i^{v_i}},s\biggr>\biggr]$, then the choices of $[L]$ are $\biggl[\biggl<r^{2\prod\limits_{i=1}^{k}p_i^{w_i}},s\biggr>\biggr]$ or $\biggl[\biggl<r^{\prod\limits_{i=1}^{k}p_i^{w_i}},s\biggr>\biggr]$. Suppose $[L]=\biggl[\biggl<r^{2\prod\limits_{i=1}^{k}p_i^{w_i}},s\biggr>\biggr]$, then using Theorem 2.5 and the fact $[H]\vee'[K]=[L]\vee'[K]$ and $[H]\wedge'[K]=[L]\wedge'[K]$, we have, $u_i=w_i$, for all $1\le i\le k$, and consequently, $[H]=[L]$, a contradiction.\par 
		Now, suppose $[L]=\biggl[\biggl<r^{\prod\limits_{i=1}^{k}p_i^{w_i}},s\biggr>\biggr]$, then by Theorem 2.5 and the fact $[H]\vee'[K]=[L]\vee'[K]$, we have $\biggl[\biggl<r^{\prod\limits_{i=1}^{k}p_i^{\min\{u_i,v_i\}}},s\biggr>\biggr]=\biggl[\biggl<r^{\prod\limits_{i=1}^{k}p_i^{\min\{w_i,v_i\}}},s\biggr>\biggr]$, i.e., $T\biggl(2\prod\limits_{i=1}^{k}p_i^{t_i-\min\{u_i,v_i\}}\biggr)=T\biggl(2\prod\limits_{i=1}^{k}p_i^{t_i-\min\{w_i,v_i\}}\biggr)$, so, $\min\{u_i,v_i\}=\min\{w_i,v_i\}$, for all $1\le i\le k$. Similarly, using Theorem 2.5 and the fact $[H]\wedge'[K]=[L]\wedge'[K]$, we have, $\biggl[\biggl<r^{2\prod\limits_{i=1}^{k}p_i^{\max\{u_i,v_i\}}},s\biggr>\biggr]=\biggl[\biggl<r^{\prod\limits_{i=1}^{k}p_i^{\max\{w_i,v_i\}}},s\biggr>\biggr]$, i.e., $T\biggl(\prod\limits_{i=1}^{k}p_i^{t_i-\max\{u_i,v_i\}}\biggr)\cup\{2\}=T\biggl(2\prod\limits_{i=1}^{k}p_i^{t_i-\max\{w_i,v_i\}}\biggr)$. Certainly, if there is $i_0$ with $t_{i_0}\neq \max\{u_{i_0},v_{i_0}\}$, then $2p_{i_0}\in T\biggl(2\prod\limits_{i=1}^{k}p_i^{t_i-\max\{w_i,v_i\}}\biggr) $ but $2p_{i_0}\notin T\biggl(\prod\limits_{i=1}^{k}p_i^{t_i-\max\{u_i,v_i\}}\biggr)\cup\{2\} $, which is not possible. Therefore, $t_i=\max\{u_i,v_i\}=\max\{w_i,v_i\}$, so, $u_i=w_i$, for all $1\le i\le k$. Thus, the sublattice of $\widetilde \Pi(D_n)$ isomorphic to $N_5$ is depicted in Figure 3.
		\begin{figure}[h]
			\centering
			\begin{tikzpicture}[scale=1.5,  every node/.style={circle, fill=black, minimum size=6pt, inner sep=0pt}]
				
				\node[label=left:${\biggl[\biggl<r^{\prod\limits_{i=1}^{k}p_i^{u_i}},s\biggr>\biggr]}$] (u2) at (0,1.4) {};
				\node[label=left:${\biggl[\biggl<r^{2\prod\limits_{i=1}^{k}p_i^{u_i}},s\biggr>\biggr]}$] (u1) at (0,0.7) {};
				\node[label=right:${\biggl[\biggl<r^{\prod\limits_{i=1}^{k}p_i^{\min\{u_i,v_i\}}},s\biggr>\biggr]}$] (u5) at (0.7,2.1) {};
				\node[label=right:${\biggl[\biggl<r^{\prod\limits_{i=1}^{k}p_i^{v_i}},s\biggr>\biggr]}$] (u3) at (1.4,1.05) {};
				\node[label=right:${\bigl[\bigl<s\bigr>\bigr]}$] (u4) at (0.7,0) {};
				
				\draw (u1) -- (u2);
				\draw (u2) -- (u5);
				\draw (u1) -- (u4);
				\draw (u4) -- (u3);
				\draw (u3) -- (u5);
			\end{tikzpicture}
			\caption{}
		\end{figure}\\
		Note that such a sublattice isomorphic to $N_5$ exists if $k\ge2$.\\
		\textbf{Subcase 4.3:} If $[H]=\biggl[\biggl<r^{\prod\limits_{i=1}^{k}p_i^{u_i}},s\biggr>\biggr]$ and $[K]=\biggl[\biggl<r^{2\prod\limits_{i=1}^{k}p_i^{v_i}},s\biggr>\biggr]$, then $[L]=\biggl[\biggl<r^{\prod\limits_{i=1}^{k}p_i^{w_i}},s\biggr>\biggr]$. Using Theorem 2.5 and the fact $[H]\vee'[K]=[L]\vee'[K]$ and $[H]\wedge'[K]=[L]\wedge'[K]$, we have $\min\{u_i,v_i\}=\min\{w_i,v_i\}$ and $\max\{u_i,v_i\}=\max\{w_i,v_i\}$, for all $1\le i\le k$. Therefore, $u_i=w_i$, for all $1\le i\le k$ and consequently, $[H]=[L]$, a contradiction.\\
		\textbf{Subcase 4.4:} If $[H]=\biggl[\biggl<r^{\prod\limits_{i=1}^{k}p_i^{u_i}},s\biggr>\biggr]$ and $[K]=\biggl[\biggl<r^{\prod\limits_{i=1}^{k}p_i^{v_i}},s\biggr>\biggr]$, then $[L]=\biggl[\biggl<r^{\prod\limits_{i=1}^{k}p_i^{w_i}},s\biggr>\biggr]$. By a similar argument as in Subcase 4.3, we get a contradiction. 
	\end{proof}
	\begin{remark}
		For the case $n=2\prod\limits_{i=1}^{k}p_i^{t_i}$ with atleast one $t_i\ne 0$, Theorem 2.9 gives a complete description of the sublattice isomorphic to $N_5$ in $\widetilde\Pi(D_n)$ as depicted in Figure 3. 
	\end{remark}
	In the light of Theorem 2.5, 2.6, 2.8 and 2.9, we conclude the following.
	\begin{theorem}
		The lattice $\widetilde \Pi(D_n)$ is modular if and only if $n$ is one of the following:
		\begin{enumerate}
			\item $n=2^\alpha $, where $\alpha $ is a non-negative integer.
			\item  $n=\prod\limits_{i=1}^{k}p_i^{t_i}$, where $p_i$'s are distinct odd primes and $t_i$'s are non-negative integers. 
			\item $n=2p^{\alpha}$, where $p$ is an odd prime and $\alpha\ge1$.
		\end{enumerate}
	\end{theorem}
	\begin{remark}
		Theorem 2.8 gurantees that for any choice of $n$, the lattice $\widetilde\Pi(D_n)$ does not contain a sublattice isomorphic to $M_3$. Therefore, from Theorem 2.9, $\widetilde\Pi(D_n)$ is modular if and only if it is distributive.
	\end{remark}
	\section*{Acknowledgement}
	Tushar Halder is thankful to CSIR for financial assistance in the form of Junior Research Fellowship (JRF), bearing the File Number: 09/0414(22038)/2025-EMR-I.
	
	\bibliographystyle{acm}
	\bibliography{References}
	
\end{document}